\documentclass[11pt]{article} 

\usepackage[T1]{fontenc}
\usepackage[utf8]{inputenc} 

\usepackage[margin=2cm, top=1.5in, bottom=1.5in]{geometry} 

\usepackage{booktabs} 
\usepackage{array} 
\usepackage{paralist} 
\usepackage{verbatim} 
\usepackage{subfig} 
\usepackage{mathtools}
\usepackage{amssymb}
\usepackage{amsthm}
\usepackage{tikz-cd}
\usepackage{mathrsfs}
\usepackage{enumitem}
\usepackage[colorlinks, allcolors=blue]{hyperref}
\usepackage{appendix}
\usepackage{stmaryrd}
\usepackage{titlesec}
\usepackage{bm}
\usepackage{nicefrac}
\usepackage[backend=bibtex,style=numeric,sorting=nyvt]{biblatex}
\usepackage[nottoc,notlof,notlot]{tocbibind} 
\usepackage[titles,subfigure]{tocloft} 
\usepackage{sectsty}
\usepackage{fancyhdr} 
\usepackage{graphicx} 
\newcommand{\8}{\infty}

\newcommand{\SO}{\operatorname{SO}}

\newcommand{\SL}{\operatorname{SL}}

\newcommand{\Rtil}{\widetilde{R}}

\newcommand{\Etil}{\widetilde{E}}

\newcommand{\thetil}{\widetilde{\theta}}
\newcommand{\nabtil}{\widetilde{\nabla}}

\newcommand{\Omegtil}{\widetilde{\Omega}}

\newcommand{\sorth}{\mathfrak{so}}

\newcommand{\D}{\mathbb{D}}
\newcommand{\HH}{\mathbb{H}}
\newcommand{\R}{\mathbb{R}}

\newcommand{\Q}{\mathbb{Q}}

\newcommand{\p}{\mathfrak{p}}

\newcommand{\g}{\mathfrak{g}}

\newcommand{\kfrak}{\mathfrak{k}}

\newcommand{\pr}{\operatorname{pr}}
\newcommand{\PD}{\operatorname{PD}}
\newcommand{\cohom}{H}
\newcommand{\End}{\operatorname{End}}

\newcommand{\Hom}{\operatorname{Hom}}

\newcommand{\Scal}{\mathscr{S}}

\newcommand{\Zcal}{\mathcal{Z}}

\newcommand{\Disk}{\operatorname{D}}
\newcommand{\Ad}{\operatorname{Ad}}
\newcommand{\rd}{\textrm{rd}}
\newcommand{\Th}{\operatorname{Th}}

\newcommand*{\transp}[2][-3mu]{\ensuremath{\mskip1mu\prescript{\smash{\mathrm t\mkern#1}}{}{\mathstrut#2}}}%
\newcommand{\hooklongrightarrow}{\lhook\joinrel\longrightarrow}

\newcommand\restr[2]{{
  \left.\kern-\nulldelimiterspace 
  #1 
  \vphantom{\big|} 
  \right|_{#2} 
  }}
\newtheoremstyle{mytheoremstyle} 
    {2em}                    
    {2em}                    
    {\itshape}                   
    {}                           
    {\normalsize \bfseries \scshape}                   
    {.}                          
    {0,5em}                       
    {}  
    
\theoremstyle{mytheoremstyle}

\newtheorem{thm}{Theorem}[section]
\newtheorem*{thm*}{Theorem}

\newtheorem*{cor*}{Corollary}
\newtheorem{lem}[thm]{Lemma}
\newtheorem{prop}[thm]{Proposition}

\theoremstyle{remark}
\newtheorem{rmk}{Remark}[section]

\numberwithin{equation}{section}
\sectionfont{\normalfont}
\subsectionfont{\normalfont\itshape}
\subsubsectionfont{\normalfont\itshape}

\title{\large The Kudla-Millson form via the Mathai-Quillen formalism}
\author{\normalsize Romain Branchereau}

\begin{document}
\maketitle

\abstract{In \cite{km2}, Kudla and Millson constructed a  $q$-form $\varphi_{KM}$ on an orthogonal symmetric space using Howe's differential operators. It is a crucial ingredient in their theory of theta lifting. This form can be seen as a Thom form of a real oriented vector bundle. In \cite{mq} Mathai and Quillen constructed a {\em canonical} Thom form and we show how to recover the Kudla-Millson form via their construction. A similar result was obtained by \cite{garcia} for signature $(2,q)$ in case the symmetric space is hermitian and we extend it to an arbitrary signature.}

\section{Introduction}

Let $(V,Q)$ be a quadratic space over $\Q$ of signature $(p,q)$ and let $G$ be its orthogonal group. Let $\D$ be the space of {\em oriented} negative $q$-planes in $V(\R)$ and $\D^+$ one of its connected components. It is a Riemannian manifold of dimension $pq$ and an open subset of the Grassmannian. The Lie group $G(\R)^+$ is the connected component of the identity and acts transitively on $\D^+$. Hence we can identify $\D^+$ with $G(\R)^+ /K $, where $K$ is a compact subgroup of $G(\R)^+$ and is isomorphic to $\SO(p)\times \SO(q)$. Moreover let $L$ be a lattice in $ V(\Q)$ and $\Gamma$ be a torsion free subgroup of $G(\R)^+$ preserving $L$. 

For every vector $v$ in $V(\R)$ there is a totally geodesic submanifold $\D^+_v$ of codimension $q$ consisting of all the negative $q$-planes that are orthogonal to $v$. Let $\Gamma_v$ denote the stabilizer of $v$ in $\Gamma$. We can view $ \Gamma_v \backslash \D^+$ as a rank $q$ vector bundle over $\Gamma_v \backslash \D_v^+$,  so that the natural embedding $\Gamma_v \backslash \D_v^+ $ in $ \Gamma_v \backslash \D^+$ is the zero section. In \cite{km2}, Kudla and Millson constructed a closed $G(\R)^+$-invariant differential form
\begin{align}
    \varphi_{KM} \in \left [ \Omega^q(\D^+) \otimes \Scal(V(\R)) \right ]^{G(\R)^+},
\end{align}
where $G(\R)^+$ acts on the Schwartz space $\Scal(V(\R))$ from the left by $(gf)(v) \coloneqq f(g^{-1}v)$  and on $\Omega^q(\D^+) \otimes \Scal(V(\R))$ from the right by $g \cdot (\omega \otimes f) \coloneqq g^\ast \omega \otimes (g^{-1} f)$. In particular $\varphi_{KM}(v)$ is a $\Gamma_v$-invariant form on $\D^+$. The main property of the Kudla-Millson form is its Thom form property: if $\omega$ in $\Omega_c^{pq-q}(\Gamma_v \backslash \D^+)$ is a compactly supported form, then
\begin{align} 
    \int_{\Gamma_v \backslash \D^+} \varphi_{KM}(v) \wedge \omega = 2^{-\frac{q}{2}}e^{-\pi Q(v,v)}\int_{\Gamma_v \backslash \D^+_v} \omega.
\end{align}
Another way to state it is to say that in cohomology we have
\begin{align} \label{thom11}
[\varphi_{KM}(v)] = 2^{-\frac{q}{2}}e^{-\pi Q(v,v)} \PD(\Gamma_v \backslash \D^+_v) \in \cohom^q \left (\Gamma_v \backslash \D^+ \right ),
\end{align}where $\PD(\Gamma_v \backslash \D^+_v)$ denotes the Poincaré dual class to $\Gamma_v \backslash\D^+_v$.

\paragraph{Kudla-Millson theta lift.} In order to motivate the interest in the Kudla-Millson form, let us briefly recall how it is used to construct a theta correspondence between certain cohomology classes and modular forms. Let $\omega$ be the Weil representation of $\SL_2(\R)$ in $\Scal(V(\R))$. We extend it to a representation in $\Omega^q(\D^+) \otimes \Scal(V(\R))$ by acting in the second factor of the tensor product. Building on the work of \cite{weil}, Kudla and Millson \cite{km3, KMIHES} used their differential form to construct the theta series
\begin{align}
    \Theta_{KM}(\tau) \coloneqq y^{-\frac{p+q}{4}} \sum_{v \in L} \Bigl (\omega(g_\tau)\varphi_{KM} \Bigr )(v) \in \Omega^q(\D^+),
\end{align}
where $\tau=x+iy$ is in $\HH$ and $g_\tau$ is the matrix $\begin{pmatrix} \sqrt{y} & x\sqrt{y}^{-1} \\ 0 & \sqrt{y}^{-1} \end{pmatrix}$ in $\SL_2(\R)$ that sends $i$ to $\tau$ by Möbius transformation. This form is $\Gamma$-invariant, closed and holomorphic in cohomology. Kudla and Millson showed that if we integrate this closed form on a {\em compact} $q$-cycle $C$ in $\Zcal_q(\Gamma \backslash \D^+)$, then
\begin{align}
    \int_C \Theta_{KM}(\tau)=c_0(C)+\sum_{n=1}^\8 \bigl < C,C_{2n} \bigr >e^{2i\pi  n\tau}
\end{align}
is a modular form of weight  $\frac{p+q}{2}$, where
\begin{align}
C_n \coloneqq \sum_{\substack{v \in \Gamma \backslash L \\ Q(v,v)=n}} C_v
\end{align}
and the {\em special cycles} $C_v$ are the images of the composition
\begin{align}
\Gamma_v \backslash \D_v^+ \hooklongrightarrow \Gamma_v \backslash \D^+ \longrightarrow \Gamma \backslash \D^+.
\end{align}
Thus, the Kudla-Millson theta series realizes a lift between the (co)-homology of $\Gamma \backslash \D^+$ and the space of weight $\frac{p+q}{2}$ modular forms.

\paragraph{The result.} Let $E$ be a $G(\R)^+$-equivariant vector bundle of rank $q$ over $\D^+$ and $E_0$ the image of the zero section. By the equivariance we also have a vector bundle $\Gamma_v \backslash E$ over $\Gamma_v \backslash \D^+$. The {\em Thom class} of the vector bundle is a characteristic class $\Th(\Gamma_v \backslash E)$ in $ \cohom^{q}(\Gamma_v \backslash E,\Gamma_v \backslash (E-E_0))$ defined by the Thom isomorphism; see Subsection \ref{thom}. A {\em Thom form} is a form representing the Thom class. It can be shown that the Thom class is also the Poincaré dual class to $\Gamma_v \backslash E_0$. Let $s_v \colon \Gamma_v \backslash \D^+ \longrightarrow \Gamma_v \backslash E$ be a section whose zero locus is $\Gamma_v \backslash \D_v^+$, then 
\begin{align}
    s_v^\ast \Th(\Gamma_v \backslash E) \in \cohom^q \left (\Gamma_v \backslash \D^+,\Gamma_v \backslash (\D^+-\D_v^+) \right ).
\end{align}
Viewing it as a class in $H^q(\Gamma_v \backslash \D^+)$ it is the Poincaré dual class of $\Gamma_v \backslash \D_v^+$. Since the Poincaré dual class is unique, property \eqref{thom11} implies that 
\begin{align} \label{thom2}
[\varphi_{KM}(v)]= 2^{-\frac{q}{2}}e^{-\pi Q(v,v)} s_v^\ast \Th(\Gamma_v \backslash E) \in \cohom^q \left (\Gamma_v \backslash \D^+\right ),
\end{align}
on the level of cohomology.

For arbitrary oriented real metric vector bundles,  Mathai and Quillen used the Chern-Weil theory to construct in \cite{mq} a canonical Thom forms on $E$. We denote by $U_{MQ}$ the canonical Thom form in $\Omega^{q}(E)$ of Mathai and Quillen. Since $U_{MQ}$ is $\Gamma$-invariant, it is also a Thom form for the bundle $\Gamma_v \backslash E$ for every vector $v$. The main result is the following.
\begin{thm*}(Theorem \ref{mainthm}) We have $\varphi_{KM}(v)= 2^{-\frac{q}{2}}e^{-\pi Q(v,v)} s_v^\ast U_{MQ}$ in $\Omega^q(\Gamma_v \backslash \D^+)$
\end{thm*}

\noindent For signature $(2,q)$, the spaces are hermitian and the result was obtained by a  similar method in \cite{garcia} using the work of Bismut-Gillet-Soulé.


\subsection*{Acknowledgements} This project is part of my thesis and I thank my advisors Nicolas Bergeron and Luis Garcia for suggesting me this topic and for their support. I was funded from the European Union’s Horizon 2020 research and innovation programme under the Marie Skłodowska-Curie grant agreement N$\textsuperscript{\underline{\scriptsize o}}$754362 \includegraphics[width=6mm]{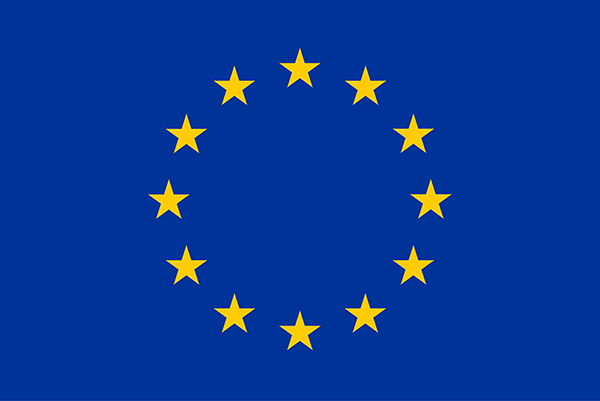}.

\section{The Kudla-Millson form}
\subsection{The symmetric space $\D$} \label{sympspacesetup}
Let $(V,Q)$ be a rational quadratic space and let $(p,q)$ be the signature of $V(\R)$. Let $e_1, \dots, e_{p+q}$ be an orthogonal basis of $V(\R)$ such that
\begin{align} \label{basis}
Q(e_\alpha , e_\alpha) =1 \quad & \textrm{for} \quad \; 1 \leq \alpha \leq p, \nonumber \\ 
Q(e_\mu , e_\mu) =-1 \quad & \textrm{for} \quad \; p+1 \leq \mu \leq p+q .   
\end{align}
Note that we will always use letters $\alpha$ and $\beta$ for indices between $1$ and $p$, and letters $\mu$ and $\nu$ for indices between $p+1$ and $p+q$. A plane $z$ in $V(\R)$ is {\em a negative plane} if $\restr{Q}{z}$ is negative definite. Let 
\begin{align}
\D \coloneqq & \left \{ z \subset V(\R) \, \vert \; \textrm{ $z$ is an oriented negative plane of dimension $q$} \right \}
\end{align}
be the set of negative oriented $q$-planes in $V(\R)$. For each negative plane there are two possible orientations, yielding two connected components $\D^+ $ and $ \D^-$ of $\D$. Let $z_0$ in $\D^+$ be the negative plane spanned by the vectors $e_{p+1}, \dots, e_{p+q}$ together with a fixed orientation. The group $G(\R)^+$ acts transitively on $\D^+$ by sending $z_0$ to $gz_0$. Let $K$ be the stabilizer of $z_0$, which is isomorphic to $\SO(p)\times \SO(q)$. Thus we have an identification
\begin{align}
G(\R)^+/K & \longrightarrow  \D^+ \nonumber \\
gK & \longmapsto gz_0.
\end{align}
For $z$ in $\D^+$ we denote by $g_z$ any element of $G(\R)^+$ sending $z_0$ to $z$.

For a positive vector $v$ in $V(\R)$ we define
\begin{align}
\D_v \coloneqq \left \{ z \in \D \; \vert \; z \subset v^\perp \right \}.
\end{align}
It is a totally geodesic submanifold of $\D$ of codimension $q$. Let $\D_v^+$ be the intersection of $\D_v$ with $\D^+$.

Let $z$ in $\D^+$ be a negative plane. With respect to the orthogonal splitting of $V(\R)$ as $z^\perp \oplus z$ the quadratic form splits as
 \begin{align}
     Q(v,v)=\restr{Q}{z^\perp}(v,v)+\restr{Q}{z}(v,v).
 \end{align}
 We define the {\em Siegel majorant at $z$} to be the positive definite quadratic form
 \begin{align}
     Q^+_z(v,v) \coloneqq \restr{Q}{z^\perp}(v,v)-\restr{Q}{z}(v,v).
 \end{align}

\subsection{The Lie algebras $\g$ and $\kfrak$} \label{liealg}
Let 
\begin{align} \g & \coloneqq \left \{ \left . \left (
\begin{array}{c c}
A & x \\
\,^tx & B
\end{array}\right )  \right \vert A \in \sorth(z_0^\perp), \; B \in \sorth(z_0), \; x \in \Hom(z_0,z_0^\perp) \right \},\end{align}
\begin{align} \kfrak & \coloneqq \left \{ \left . \left (
\begin{array}{c c}
A & 0 \\
0 & B
\end{array}\right )  \right \vert A \in \sorth(z_0^\perp), \; B \in \sorth(z_0) \right \}\end{align}
be the Lie algebras of $G(\R)^+$ and $K$ where $\sorth(z_0)$ is equal to $\sorth(q)$. The latter is the space of skew-symmetric $q$ by $q$ matrices. Similarly we have $\sorth(z_0^\perp)$ equals $\sorth(p)$. Hence we have a decomposition of $\kfrak$ as $\sorth(z_0^\perp) \oplus \sorth(z_0)$ that is orthogonal with respect to the Killing form. Let $\epsilon$ be the Lie algebra involution of $\g$ mapping $X$ to $-\transp{X}$. The $+1$-eigenspace of $\epsilon$ is $\kfrak$ and the $-1$-eigenspace is
\begin{align} \p \coloneqq \left \{ \left . \left (
\begin{array}{c c}
0 & x \\
\,^tx & 0
\end{array}\right )  \right \vert  x \in \Hom(z_0,z_0^\perp) \right \}.\end{align}
We have a decomposition of $\g $ as $\kfrak \oplus \p$ and it is orthogonal with respect to the Killing form. We can identify $\p$ with $\g/\kfrak$. Since $\epsilon$ is a Lie algebra automorphism we have that 
\begin{align}
   [\p,\p] \subset \kfrak, \qquad \qquad [\kfrak,\p ] \subset \p.
\end{align}
We identify the tangent space of $\D^+$ at $eK$ with $\p$ and the tangent bundle  $T\D^+$ with $G(\R)^+ \times_K \p$ where $K$ acts on $\p$ by the $\Ad$-representation. We have an isomorphism
\begin{align} \label{eq:sowedge2}
 T \colon \wedge^2V(\R) & \longrightarrow \g \nonumber \\
e_i \wedge e_j & \longmapsto T(e_i \wedge e_j)e_k \coloneqq Q(e_i,e_k ) e_j-Q( e_j,e_k ) e_i.
\end{align}
A basis of $\g$ is given by the set of matrices 
\begin{align}
    \left \{ \left . X_{ij} \coloneqq T(e_i \wedge e_j) \in \g \right \vert 1<i<j<p+q \right \}
\end{align}
and we denote by $\omega_{ij}$ its dual basis in the dual space $\g^\ast$. Let $E_{ij}$ be the elementary matrix sending $e_i$ to $e_j$ and the other $e_k$'s to $0$.  Then $\p$ is spanned by the matrices
\begin{align}X_{\alpha \mu} = E_{\alpha \mu}+E_{\mu \alpha}\end{align}
and $\kfrak$ is spanned by the matrices
\begin{align}
X_{\alpha \beta} & =  E_{\alpha \beta}-E_{\beta \alpha}, \nonumber \\
X_{\nu \mu} & = -E_{\nu \mu}+E_{\mu \nu}.
\end{align}


\subsection{Poincaré duals}

Let $M$ be an arbitrary $m$-dimensional real orientable manifold without boundary. The integration map yields a non-degenerate pairing \cite[Theorem.~5.11]{botu}
\begin{align} \cohom^{q}(M) \otimes_\R  \cohom_c^{m-q}(M)& \longrightarrow \R \nonumber \\
[\omega]\otimes [\eta] & \longmapsto \int_M \omega \wedge \eta,    
\end{align} where $\cohom_c(M)$ denotes the cohomology of compactly supported forms on $M$. This yields an isomorphism between $\cohom^{q}(M)$ and the dual $\cohom_c^{m-q}(M)^\ast=\Hom(\cohom_c^{m-q}(M),\R)$. If $C$ is an immersed submanifold of codimension $q$ in $M$ then $C$ defines a linear functional on $\cohom_c^{m-q}(M)$ by
\begin{align}
 \omega \longmapsto \int_C \omega.   
\end{align} Since we have an isomorphism between $\cohom_c^{m-q}(M)^\ast$ and $\cohom^{q}(M)$ there is a unique cohomology class $\PD(C)$ in $H^q(M)$ representing this functional {\em i.e. } 
\begin{align}
\int_M   \omega \wedge \PD(C) = \int_C \omega
\end{align}
 for every class $[\omega]$ in $\cohom_c^{m-q}(M)$. We call $\PD(C)$ {\em the Poincaré dual class to $C$}, and any differential form representing the cohomology class $\PD(C)$ {\em a Poincaré dual form to $C$}. 
 
\subsection{The Kudla-Millson form}
 The tangent plane at the identity $T_{eK} \D^+ $ can be identified with $\p$ and the cotangent bundle $(T\D^+)^\ast$ with $G(\R)^+ \times_K \p^\ast$, where $K$ acts on $\p^\ast$ by the dual of the $\Ad$-representation.
 The basis $e_1, \dots, e_{p+q}$ identifies $V(\R)$ with $\R^{p+q}$. With respect to this basis the Siegel majorant at $z_0$ is given by
 \begin{align}
     Q^+_{z_0}(v,v)\coloneqq\sum_{i=1}^{p+q} x_i^2.
 \end{align} Recall that $G(\R)^+$ acts on $\Scal(\R^{p+q})$ from the left by $(g \cdot f)(v)= f(g^{-1}v)$ and on $\Omega^q(\D^+) \otimes \Scal(\R^{p+q})$ from the right by $g \cdot (\omega \otimes f) \coloneqq g^\ast \omega \otimes (g^{-1} f)$. We have an isomorphism
\begin{align} \label{isoate}
    \left [ \Omega^q(\D^+) \otimes \Scal(\R^{p+q}) \right ]^{G(\R)^+} & \longrightarrow \left [ \sideset{}{^q}\bigwedge \p^\ast \otimes \Scal(\R^{p+q}) \right ]^K \nonumber \\
    \varphi & \longrightarrow \varphi_e
\end{align}
by evaluating $\varphi$ at the basepoint $eK$ in $G(\R)^+/K$, corresponding to the point $z_0$ in $\D^+$. We define the {\em Howe operator}
\begin{align}
D \colon \sideset{}{^\bullet}\bigwedge \p^\ast \otimes \Scal(\R^{p+q}) & \longrightarrow \sideset{}{^{\bullet+q}}\bigwedge \p^\ast \otimes \Scal(\R^{p+q})
\end{align}
by
\begin{align}
D & \coloneqq \frac{1}{2^q} \prod_{\mu=p+1}^{p+q}  \sum_{\alpha=1}^p A_{\alpha \mu} \otimes \left (x_\alpha - \frac{1}{2\pi}\frac{\partial}{\partial{x_\alpha}} \right )
\end{align}
where $A_{\alpha \mu}$ denotes left multiplication by $\omega_{\alpha \mu}$. The Kudla-Millson form is defined by applying $D$ to the Gaussian:

\begin{align} \label{kmdef}
    \varphi_{KM}(v)_e \coloneqq D \exp \left (-\pi Q^+_{z_0}(v,v) \right ) \in  \sideset{}{^q}\bigwedge \p^\ast \otimes \Scal(\R^{p+q}).
\end{align}
Kudla and Millson showed that this form is $K$-invariant. Hence by the isomorphism \eqref{isoate} we get a form
\begin{align}
    \varphi_{KM} \in \left [ \Omega^q(\D^+) \otimes \Scal(\R^{p+q}) \right ]^{G(\R)^+}.
\end{align}
In particular it is $\Gamma_v$-invariant and defines a form on $\Gamma_v \backslash \D^+$. It is also closed and satisfies the Thom form property: for every compactly supported form $\omega$ in $\Omega^{pq-q}_c(\Gamma_v \backslash \D^+)$ we have
\begin{align} \label{thom1}
\int_{\Gamma_v \backslash \D^+} \omega \wedge \varphi_{KM}(v) = 2^{-\frac{q}{2}}e^{-\pi Q(v,v)}\int_{\Gamma_v \backslash \D^+_v} \omega.
\end{align}


\section{The Mathai-Quillen formalism} \label{sectionmq}

We begin by recalling a few facts about principal bundles, connections and associated vector bundles. For more details we refer to \cite{berline} and \cite{kobayashi1}. The Mathai-Quillen form is defined in Subsection \ref{mqconstruction} following \cite{berline}; see also \cite{getzler}.

\subsection{$K$-principal bundles and principal connections} \label{assbundle}

Let $K$ be $\SO(p)\times \SO(q)$ as before and  $P$ be a smooth principal $K$-bundle. Let 
\begin{align}
    R \colon K \times P \longrightarrow P \nonumber \\
    (k,p) \longmapsto R_k(p)
\end{align} be the smooth right action of $K$ on $P$ and 
\begin{align}
    \pi \colon P \longrightarrow P/K
\end{align}
the projection map. For a fixed $p $ in $P$ consider the map
\begin{align}
    R_p \colon K & \longrightarrow P \nonumber \nonumber \\
    k & \longmapsto R_k(p).
\end{align}
Let $V_pP$ be the image of the derivative at the identity
\begin{align}
    d_e R_p \colon \kfrak \longrightarrow T_pP,
\end{align} which is injective. It coincides with the kernel of the differential $d_p\pi$. A vector in $V_pP$ is called a {\em vertical vector}. Using this map we can view a vector $X$ in $\kfrak$ as a vertical vector field on $P$. The space $P$ can a priori be arbitrary, but in our case we will consider either 
\begin{enumerate}
    \item $P$ is $G(\R)^+$ and $R_k$ the natural right action sending $g$ to $gk$. Then $P/K$ can be identified with $\D^+$,
    \item $P$ is $G(\R)^+ \times z_0$ and the action $R_k$ maps $(g,w)$ to $(gk,k^{-1}w)$. In this case $P/K$ can be identified with $G(\R)^+ \times_K z_0$. It is the vector bundle associated to the principal bundle $G(\R)^+$ as defined below.
\end{enumerate}

A {\em principal $K$-connection} on $P$ is a $1$-form $\theta_P $ in $\Omega^1(P, \kfrak)$ such that
\begin{itemize}
\item $\iota_X \theta_P = X$ \quad for any $X $ in $\kfrak$,
\item $R_k^\ast \theta_P=Ad(k^{-1}) \theta_P \quad$ for any $k$ in $K$,
\end{itemize}
where $\iota_X$ is the interior product
\begin{align}
    \iota_X \colon & \Omega^k(P) \longrightarrow \Omega^k(P) \nonumber \\
    \omega & \longmapsto (\iota_X \omega)(X_1, \dots, X_{p-1}) \coloneqq \omega(X,X_1,\dots, X_{p-1}).
\end{align}
and we view $X$ as a vector field on $P$. Geometrically these conditions imply that the kernel of $\theta_P$ defines a horizontal subspace  of $TP$ that we denote by $HP$. It is a complement to the vertical subspace {\em i.e.} we get a splitting of $T_pP$ as $V_pP \oplus H_pP$.

Let $\g$ be the Lie algebra of $G(\R)^+$ and let $p$ be the orthogonal projection from $\g$ on $\kfrak$. After identifying $\g^\ast$ with the space $\Omega^1(G(\R)^+)^{G(\R)^+}$ of $G(\R)^+$-invariant forms we define a natural $1$-form
\begin{align}
 \sum_{1 \leq i<j \leq p+q} \omega_{ij} \otimes X_{ij} \in \Omega^1(G(\R)^+) \otimes \g
\end{align}
called the {\em Maurer-Cartan form}, where $X_{ij}$ is the basis of $\g$ defined earlier and $\omega_{ij}$ its dual in $\g^\ast$.  After projection onto $\kfrak$ we get a form
\begin{align} \label{connection1} \theta \coloneqq p \left (\sum_{1 \leq i<j \leq p+q} \omega_{ij} \otimes X_{ij} \right ) \in \Omega^1(G(\R)^+) \otimes \kfrak \end{align}
where we identify $\Omega^1(G(\R)^+, \kfrak)$ with $\Omega^1(G(\R)^+) \otimes \kfrak$. A direct computation shows that it is a principal $K$-connection on $P$ when $P$ is $G(\R)^+$.

If $P$ is $G(\R)^+ \times z_0$ then the projection 
\begin{align}
    \pi \colon G(\R)^+ \times z_0 \longrightarrow G(\R)^+
\end{align}induces a pullback map 
\begin{align}
    \pi^\ast \colon \Omega^1(G(\R)^+) \longrightarrow \Omega^1(G(\R)^+ \times z_0).
\end{align}The form \begin{align} \label{connection2}
  \thetil \coloneqq \pi^\ast\theta \in \Omega^1(G(\R)^+ \times z_0)\otimes \kfrak  
\end{align}
is a principal connection on $G(\R)^+ \times z_0$.

\subsection{The associated vector bundles}
Since $z_0$ is preserved by $K$ we have an orthogonal $K$-representation 
\begin{align} \label{representation}
    \rho \colon K & \longrightarrow \SO(z_0) \nonumber \\
    k & \longmapsto \rho(k)w \coloneqq \restr{k}{z_0}w,
\end{align}
where we will usually simply write $kw$ instead of $\restr{k}{z_0}w$. We can consider the {\em associated vector bundle} $P \times_K z_0$ which is the quotient of $P \times z_0$ by $K$, where $K$ acts by sending $(p,w)$ to $ (R_k(p), \rho(k)^{-1}w)$. Hence an element $[p,w]$ of $P \times_K z_0$ is an equivalence class where the equivalence relation identifies $(p,w)$ with $(R_k(p), \rho(k)^{-1}w)$. This is a vector bundle over $P/K$ with projection map sending $[p,w]$ to $\pi(p)$. Let $\Omega^i(P/K,P \times_Kz_0)$ be the space of $i$-forms valued in $P \times_Kz_0$, when $i$ is zero it is the space of smooth sections of the associated bundle.

In the two cases of interest to us we define
\begin{align}
    E  & \coloneqq G(\R)^+ \times_K z_0, \nonumber \\
    \Etil  & \coloneqq \left ( G(\R)^+ \times z_0 \right ) \times_K z_0.
\end{align}
Note that in both cases $P$ admits a left action of $G(\R)^+$ and that the associated vector bundles are $G(\R)^+$-equivariant. Morever it is a Euclidean bundle, equipped with the inner product 
\begin{align}
    \langle v , w \rangle \coloneqq -\restr{Q}{z_0}(v,w)
\end{align}
on the fiber. Let $ \Omega^i(P,z_0)$ be the space of $z_0$-valued differential $i$-forms on $P$. A differential form $\alpha$ in $\Omega^i(P,z_0)$ is said to be {\em horizontal} if $\iota_X\alpha$ vanishes for all vertical vector fields $X$. There is a left action of $K$ on a differential form $\alpha$ in $\Omega^i(P,z_0)$ defined by 
\begin{align}
k\cdot \alpha \coloneqq \rho(k)(R_k^\ast \alpha),  
\end{align}
and $\alpha$ is  $K$-invariant if it satisfies $k\cdot \alpha= \alpha$ for any $k$ in $K$ {\em i.e.} we have $R_k^\ast \alpha = \rho(k^{-1}) \alpha$. We write $\Omega^i(P, z_0)^K$ for the space of $K$-invariant $z_0$-valued forms on $P$. Finally a form that is horizontal and $K$-invariant is called a {\em basic form} and the space of such forms is denoted by $\Omega^i(P,z_0)_{bas}$.

Let $X_1, \dots, X_N$ be tangent vectors of $P/K$ at $\pi(p)$ and $\widetilde{X}_i$ be tangent vectors of $P$ at $p$ that satisfy $d_p\pi(\widetilde{X}_i)=X_i$. There is a map
\begin{align}
    \Omega^i(P,z_0)_{bas} & \longrightarrow \Omega^i(P/K,P \times_K z_0)  \nonumber \\
    \alpha & \longmapsto \omega_\alpha
\end{align}
defined by 
\begin{align}
    \restr{\omega_\alpha}{\pi(p)}(X_1 \wedge \cdots \wedge X_N)=\restr{\alpha}{p}(\widetilde{X}_1 \wedge \cdots \wedge \widetilde{X}_N).
\end{align}
\begin{prop} \label{basicforms} The map is well-defined and yields an isomorphism between $\Omega^i(P/K,P \times_K z_0)$ and $\Omega^i(P,z_0)_{bas}$.
In particular if $z_0$ is $1$-dimensional then $\Omega^i(P/K)$ is isomorphic to $\Omega^i(P)_{bas}$.
\end{prop}
\begin{proof} In the case where $i$ is zero the horizontally condition is vacuous and the isomorphism simply identifies $\Omega^0(P/K,P \times_K z_0)$ with $\Omega^0(P,z_0)^K$. We have a map
\begin{align}
    \Omega^0(P,z_0)^K & \longrightarrow \Omega^0(P/K,P \times_K z_0) \nonumber \\
    f & \longmapsto s_f(\pi(p)) \coloneqq [p,f(p)],
\end{align}
which is well defined since 
\begin{align} \label{equivariant}
 f(R_k(p))=\rho(k)^{-1}f(p).  
\end{align} Conversely every smooth section $s$ in $\Omega^0(P/K,P \times_K z_0)$ is given by \begin{align}
    s(\pi(p))=[p,f_s(p)]
\end{align} for some smooth function $f_s$ in $\Omega^0(P,z_0)^K$. The map sending $s$ to $f_s$ is inverse to the previous one. The proof is similar for positive $i$.
\end{proof}

\subsection{Covariant derivatives} A {\em covariant derivative} on the vector bundle $P \times_K z_0$ is a differential operator
\begin{align}
    \nabla_P \colon \Omega^0(P/K,P \times_K z_0) & \longrightarrow \Omega^1(P/K,P \times_K z_0)
\end{align}
such that for every smooth function $f$ in $C^\8(P/K)$ we have
\begin{align}
\nabla_P(fs)=df\otimes s + f \nabla_P(s).    
\end{align}
The inner product on $P \times_Kz_0$ defines a pairing
\begin{align}
\Omega^i(P/K,P \times_K z_0) \times \Omega^j(P/K,P \times_K z_0) & \longrightarrow \Omega^{i+j}(P/K) \nonumber \\
   ( \omega_1 \otimes s_1 , \omega_2 \otimes s_2 ) & \longmapsto \langle \omega_1 \otimes s_1 , \omega_2 \otimes s_2 \rangle = \omega_1 \wedge \omega_2 \langle s_1,s_2 \rangle,
\end{align}
and we say that the derivative is compatible with the metric if
\begin{align}
d\langle s_1,s_2 \rangle = \langle \nabla_Ps_1,s_2 \rangle +\langle s_1,\nabla_Ps_2 \rangle
\end{align}
for any two sections $s_1$ and $s_2$ in $\Omega^0(P/K,P\times_Kz_0)$.
There is a covariant derivative that is induced by a principal connection $\theta_P$ in $\Omega^1(P) \otimes \kfrak$ as follows. The derivative of the representation gives a map
\begin{align}
    d\rho \colon \kfrak \longrightarrow \sorth(z_0) \subset \End(z_0),
\end{align} 
which we also denote by $\rho$ by abuse of notation. Note that for the representation \eqref{representation} this is simply the map
\begin{align}
    \rho \colon \kfrak & \longrightarrow \sorth(z_0) \nonumber \\
    X & \longmapsto \restr{X}{z_0}
\end{align}
since $\kfrak$ splits as $\sorth(z_0^\perp) \oplus \sorth(z_0)$. Composing the principal connection with $\rho$ defines an element
\begin{align}
    \rho(\theta_P) \in \Omega^1(P,\sorth(z_0)).
\end{align}
In particular, if $s$ is a section of $P \times_K z_0$ then we can identify it with a $K$-invariant smooth map $f_s$ in $\Omega^0(P, z_0)^K$. Since $\rho(\theta_P)$ is a $\sorth(z_0)$-valued form  and $\sorth(z_0)$ is a subspace of $\End(z_0)$ we can define 
\begin{align}
    df_s+\rho(\theta_P) \cdot f_s \in \Omega^1(P, z_0).
\end{align}

\begin{lem} The form $df_s+\rho(\theta_P) \cdot f_s$ is basic, hence gives a $P \times_K z_0$-valued form on $P/K$. Thus $d+\rho(\theta_P)$ defines a covariant derivative on $P \times_K z_0$. Moreover, it is compatible with the metric.
\end{lem}
\begin{proof}
See \cite[p.~24]{berline}. For the compatibility with the metric, it follows from the fact that the connection $\rho(\theta_P)$ is valued in $\sorth(z_0)$ that
\begin{align}\langle \rho(\theta_P) f_{s_1},f_{s_2} \rangle + \langle  f_{s_1}, \rho(\theta_P)f_{s_2} \rangle =0. \end{align}
Hence if we denote by $\nabla_P$ is the covariant derivative defined by $d+\rho(\theta_P)$ then
\begin{align}\langle \nabla_Ps_1,s_2 \rangle +\langle s_1,\nabla_Ps_2 \rangle = \langle df_{s_1},f_{s_2} \rangle + \langle  f_{s_1}, df_{s_2} \rangle =d \langle f_{s_1},f_{s_2}\rangle=d\langle s_1,s_2 \rangle. \end{align}
\end{proof}
\noindent Let us denote by $\nabla_P$ the covariant derivative $d+\rho(\theta_P)$. It can be extended to a map
\begin{align} \label{extconnec}
    \nabla_P \colon \Omega^i(P/K,P \times_K z_0) & \longrightarrow \Omega^{i+1}(P/K,P \times_K z_0)
\end{align}
by setting
\begin{align}
    \nabla_P(\omega \otimes s) \coloneqq d\omega \otimes s+(-1)^{i}\omega \wedge \nabla_P(s),
\end{align}
where
\begin{align}
    \omega \otimes s \in \Omega^i(P/K) \otimes \Omega^0(P/K,P \times_K z_0) \simeq \Omega^i(P/K,P \times_K z_0).
\end{align}
We define the {\em curvature} $R_P$ in $\Omega^2(P,\kfrak)$ by
\begin{align}
    R_P(X,Y) \coloneqq [\theta_P(X),\theta_P(Y)]-\theta_P([X,Y])
\end{align}
for two vector fields $X$ and $Y$ on $P$. It is basic by \cite[Proposition.~1.13]{berline} and composing with $\rho$ gives an element
\begin{align}
    \rho(R_P)\in \Omega^2(P,\sorth(z_0))_{bas},
\end{align}
so that we can view it as an element in $\Omega^2(P/K,P \times_K \sorth(z_0))$ where $K$ acts on $\sorth(z_0)$ by the $\Ad$-representation.
For a section $s$ in $\Omega^0(P/K,P \times_K z_0)$ we have \cite[Proposition.~1.15]{berline}
\begin{align}
    \nabla_P^2s=\rho(R_p)s\in \Omega^2(P/K,P\times_K z_0).
\end{align} From now on we denote by $\nabla$ and $\nabtil$ the covariant derivatives on $E$ and $\Etil$ associated to $\theta$ and $\thetil$ defined in \eqref{connection1} and \eqref{connection2}. Let $R$ and $\Rtil$ be their respectives curvatures.

\subsection{Pullback of bundles}\label{remk1}

The pullback of $E$ by the projection map gives a canonical bundle 
\begin{align}
\pi^\ast E \coloneqq \left \{ (e,e') \in E \times E \, \vert \, \pi(e)=\pi(e') \right \}    
\end{align} over $E$. We have the following diagram
\begin{equation}
\begin{tikzcd}
\pi^\ast E \arrow[d] \arrow[r] & E\arrow[d,"\pi"] \\
E  \arrow[r,"\pi"] & \D^+. 
\end{tikzcd}
\end{equation}
The projection induces a pullback of the sections 
\begin{align}
    \pi^\ast \colon \Omega^i(\D,E) \longrightarrow \Omega^i(E,\Etil).
\end{align} We can also pullback the covariant derivative $\nabla$ to a covariant derivative 
\begin{align}
    \pi^\ast \nabla \colon \Omega^0(E,\pi^\ast E) \longrightarrow \Omega^1(E,\pi^\ast E)
\end{align} on $\pi^\ast E$. It is characterized by the property
\begin{align}
(\pi^\ast \nabla)(\pi^\ast s)=\pi^\ast (\nabla s).
\end{align}

\begin{prop}
The bundles $\Etil$ and $\pi^\ast E$ are isomorphic, and this isomorphism identifies $\nabtil$ and $\pi^\ast \nabla$.
\end{prop}
\begin{proof}
By definition $([g_1,w_1],[g_2,w_2])$ are elements of $ \pi^\ast E$ if and only if $g^{-1}_1g_2$ is in $K$. We have a $G(\R)^+$-equivariant morphism
\begin{align}
    \pi^\ast E & \longrightarrow \Etil \nonumber \\
    ([g_1,w_1],[g_2,w_2]) & \longrightarrow [(g_1,g_1^{-1}g_2w_2),w_1].
\end{align}
This map is well defined and has as inverse
\begin{align}
     \Etil & \longrightarrow \pi^\ast E \nonumber \\
     [(g,w_1),w_2] & \longrightarrow ([g,w_2],[g,w_1]).
\end{align}
The second statement follows from the fact that $\thetil$ is $\pi^\ast \theta$.
\end{proof}

\subsection{A few operations on the vector bundles}

We extend the $K$-representation $z_0$ to $\bigwedge^j z_0$ by
\begin{align}
    k(w_1 \wedge \cdots \wedge w_j)=(kw_1) \wedge \cdots \wedge (k w_j).
\end{align}We consider the bundles $P \times_K \wedge^j z_0$ and $ P \times_K \wedge z_0$ over $P/K$, where $\bigwedge z_0$ is defined as $ \bigoplus_i \bigwedge^iz_0$. Denote the space of differential forms valued in $P \times_K \wedge^j z_0$ by
\begin{align}
\Omega_P^{i,j} \coloneqq \Omega_P^i(P/K,P \times_K \wedge^j z_0)=\Omega_P^i(P/K) \otimes \Omega^0(P/K,P \times_K \wedge^j z_0).
\end{align}
The total space of differential forms 
\begin{align}
 \Omega(P/K,P\times_K \wedge z_0) = \bigoplus_{i,j} \Omega_P^{i,j}
\end{align}
 is an (associative) bigraded $C^\8(P/K)$-algebra where the product is defined by
\begin{align}
\wedge \colon \Omega_P^{i,j} \times \Omega_P^{k,l}& \longrightarrow \Omega_P^{i+k,j+l} \nonumber \\
(\omega \otimes s, \eta \otimes t) & \longmapsto (\omega \otimes s) \wedge ( \eta \otimes t) \coloneqq (-1)^{jk}(\omega \wedge \eta) \otimes (s \wedge t).
\end{align}
This algebra structure allows us to define an {\em exponential map} by
\begin{align}
\exp \colon \Omega(P/K,P\times_K \wedge z_0) & \longrightarrow \Omega(P/K,P\times_K \wedge z_0) \nonumber \\
\omega & \longmapsto \exp(\omega) \coloneqq \sum_{k \geq 0} \frac{\omega^k}{k!}
\end{align}
where $\omega^k$ is the $k$-fold wedge product $\omega \wedge \cdots \wedge \omega$.
\begin{rmk} \label{commutative} Suppose that $\omega$ and $\eta$ commute. Then the binomial formula
\begin{align}
    (\omega+\eta)^k=\sum_{l=0}^k \binom{k}{l}\omega^l\eta^{k-l}
\end{align}
holds and one can show that $\exp(\omega+\eta)=\exp(\omega)+\exp(\eta)$ in the same way as for the real exponential map. In particular the diagonal subalgebra $\bigoplus \Omega_P^{i,i}$ is a commutative since for two forms $\omega$ and $\eta$ in $\Omega_P$ we have 
\begin{align}
    \omega \wedge \eta =(-1)^{\deg(\omega)+\deg(\eta)} \eta \wedge \omega
\end{align} and similarly for two sections $s$and $t$ in $\Omega^0(P/K,P \times_K z_0)$.
\end{rmk}
\noindent The inner product $\langle - , - \rangle$ on $z_0$ can be extended to an inner product on $\bigwedge z_0$ by
\begin{equation} \label{extprod}
\langle \;v_1 \wedge \cdots \wedge v_k \; ,\; w_1 \wedge \cdots \wedge w_l \; \rangle \coloneqq \left \{
\begin{array}{l c r}
0 & \textrm{if} & k \neq l ,\\
\det \langle v_i,w_j \rangle_{i,j} & \textrm{if} & k = l.
\end{array}
\right .
\end{equation}
 If $e_1, \dots, e_q$ is an orthonormal basis of $z_0$, then the set 
\begin{align} \{ e_{i_1} \wedge \cdots \wedge e_{i_k} \; \vert \; 1 \leq k \leq q, \; i_1 < i_2< \cdots <i_k \}\end{align}
is an orthonormal basis of $\bigwedge z_0$. We define the {\em Berezin integral} $\int^B$ to be the orthogonal projection onto the top dimensional component, that is the map
\begin{align} 
\int^B \colon \bigwedge z_0& \longrightarrow \R \nonumber \\
w & \longmapsto \langle w \; ,  \;e_1\wedge \cdots \wedge e_q \rangle.
\end{align}
The Berezin integral can then be extended to
\begin{align}
\int^B \colon \Omega(P/K,P\times_K \wedge z_0) & \longrightarrow \Omega(P/K) \nonumber \\
\omega \otimes s & \longmapsto \omega \int^B s
\end{align}
where $\int^B s$ in $C^\8(P/K)$ is the composition of the section with the Berezinian in every fiber. Let $s_1, \dots, s_q$ be a local orthonormal frame of $P \times_K z_0$. Then $s_1 \wedge \cdots \wedge s_q$ is in $\Omega^0(P/K,\wedge^q P \times_K z_0)$ and defines a global section. Hence for $\alpha$ in $\Omega(P/K,P\times_K \wedge z_0)$ we have
\begin{align}
\int^B\alpha=\langle \alpha, s_1 \wedge \cdots \wedge s_q \rangle.
\end{align}
Finally, for every section $s$ in $\Omega^{0,1}$ we can define the {\em contraction}
\begin{align}
i(s) \colon \Omega_P^{i,j} & \longrightarrow \Omega_P^{i,j-1} \nonumber \\
\omega \otimes s_1 \wedge \cdots \wedge s_j & \longmapsto \sum_{k=1}^j (-1)^{i+k-1}\langle s,s_k \rangle \omega \otimes s_1\wedge \cdots \wedge \widehat{s_k} \wedge \cdots \wedge s_j
\end{align}
and extended by linearity, where the symbol $ \, \widehat{\cdot} \,$ means that we remove it from the product. Note that when $j$ is zero then $i(s)$ is defined to be zero. The contraction $i(s)$ defines a derivation on $\oplus \Omegtil^{i,j}$ that satisfies
\begin{align}
    i(s)(\alpha \wedge \alpha')=(i(s) \alpha) \wedge \alpha'+(-1)^{i+j}\alpha \wedge (i(s)\alpha')
\end{align}
for $\alpha $ in $\Omegtil^{i,j}$ and $\alpha'$ in $\Omegtil^{k,l}$.

\subsection{Thom forms} \label{thom}

We denote by $E$ the bundle $G(\R)^+\times_K z_0$. On the fibers of the bundle we have the inner product given by $\langle w,w' \rangle \coloneqq -Q(w,w')$. Let $v$ be arbitrary vector in $L$ and $\Gamma_v$ its stabilizer. Since the bundle is $G(\R)^+$-equivariant we have a bundle 
\begin{align}
\Gamma_v \backslash E \longrightarrow \Gamma_v \backslash \D^+,   
\end{align} and let $\Disk(\Gamma_v \backslash E)$ be the closed disk bundle. If we have a closed $(q+i)$-form on $\Gamma_v \backslash E$ whose support is contained in $\Disk(\Gamma_v \backslash E)$, then it has compact support in the fiber and represents a class in $\cohom^{q+i}(\Gamma_v \backslash E,\Gamma_v \backslash E-\Disk(\Gamma_v \backslash E))$. The cohomology group $\cohom^{\bullet}(\Gamma_v \backslash E,\Gamma_v \backslash E-\Disk(\Gamma_v \backslash E))$ is equal to the cohomology group $\cohom^{\bullet}(\Gamma_v \backslash E,\Gamma_v \backslash (E-E_0))$ that we used in the introduction, where $E_0$ is the zero section. Fiber integration induces an isomorphism on the level of cohomology
\begin{align}
\Th \colon \cohom^{q+i}(\Gamma_v \backslash E,\Gamma_v \backslash E-\Disk(\Gamma_v \backslash E)) & \longrightarrow \cohom^i(\Gamma_v \backslash \D^+) \nonumber \\
[\omega] & \longmapsto \int_{\rm fiber} \omega
\end{align}
known as the {\em Thom isomorphism} \cite[Theorem.~6.17]{botu}. When $i$ is zero then $H^i(\Gamma_v \backslash \D^+)$ is $\R$ and we call the preimage of $1$
\begin{align}
\Th(\Gamma_v \backslash E) \coloneqq \Th^{-1}(1) \in \cohom^{q}(\Gamma_v \backslash E,\Gamma_v \backslash E-\Disk(\Gamma_v \backslash E))
\end{align}
the {\em Thom class}. Any differential form representating this class is called a {\em Thom form} , in particular every closed $q$-form on $\Gamma_v \backslash E$ that has compact support in every fiber and whose integral along every fiber is $1$ is a Thom form. One can also view the Thom class as the Poincaré dual class of the zero section $E_0$ in $E$, in the same sense as for \eqref{thom1}. 

Let $\omega$ in $\Omega^j(E)$ be a form on the bundle and let $\omega_z$ be its restriction to a fiber $E_z=\pi^{-1}(z)$ for some $z$ in $\D^+$. After identifying $z_0$ with $\R^q$ we see $\omega_z$ as an element of $C^\8(\R^q) \otimes \wedge^j(\R^q)^\ast$. We say that $\omega$ is {\em rapidly decreasing in the fiber} if $\omega_z$ lies in $\Scal(\R^q) \otimes \wedge^j(\R^q)^\ast$ for every $z$ in $\D^+$. We write $\Omega^j_{\rd}(E)$ for the space of such forms.

Let $\Omega^\bullet_{\rd}(\Gamma_v \backslash E)$ be the complex of rapidly decreasing forms in the fiber. It is isomorphic to the complex $\Omega^\bullet_{\rd}(E)^{\Gamma_v}$ of rapidly decreasing $\Gamma_v$-invariant forms on $E$. Let $\cohom_{\rd}(\Gamma_v \backslash E)$ the cohomology of this complex. The map 
\begin{align}
    h \colon  \Gamma_v \backslash E & \longrightarrow \Gamma_v \backslash E \nonumber \\
    w & \longrightarrow \frac{w}{\sqrt{1-\lVert {w} \rVert^2}}
\end{align}
is a diffeomorphism from the open disk bundle $\Disk(\Gamma_v \backslash E)^\circ$ onto $\Gamma_v \backslash E$. It induces an isomorphism by pullback
\begin{align}
    h^\ast \colon \cohom_{\rd}(\Gamma_v \backslash E) \longrightarrow \cohom(\Gamma_v \backslash E,\Gamma_v \backslash E-\Disk(\Gamma_v \backslash E)),
\end{align}
which commutes with the fiber integration. Hence we have the following version of the Thom isomorphism
\begin{align}
\cohom_{\rd}^{q+i}(\Gamma_v \backslash E) & \longrightarrow \cohom^i(\Gamma_v \backslash \D^+).
\end{align}
The construction of Mathai and Quillen produces a Thom form
\begin{align}
    U_{MQ} \in \Omega_{\rd}^q(E)
\end{align}
which is $G(\R)^+$-invariant (hence $\Gamma_v$-invariant) and closed. We will recall their construction in the next section.

\subsection{The Mathai-Quillen construction} \label{mqconstruction}

As earlier let $\Etil$ be the bundle $(G(\R)^+ \times z_0) \times_K z_0$. Let $ \wedge^j \tilde{E} $ be the bundle $(G(\R)^+ \times z_0) \times_K \wedge^j z_0$ and
\begin{align}
\Omega^{i,j} & \coloneqq \Omega^i(\D^+,\wedge^jE) \nonumber \\
\Omegtil^{i,j} &\coloneqq \Omega^i(E,\wedge^j\Etil).
\end{align} 
First consider the tautological section ${\bf s}$ of $E$ defined by 
\begin{align}
 {\bf s}[g,w] \coloneqq [(g,w),w] \in \Etil.   
\end{align}
This gives a canonical element ${\bf s}$ of $\Omegtil^{0,1}$. Composing with the norm induced from the inner product we get an element $\lVert {\bf s} \rVert^2$ in $\Omegtil^{0,0}$. 

The representation $\rho$ on $z_0$ induces a representation on $\wedge^iz_0$ that we also denote by $\rho$. The derivative at the identity gives a map 
\begin{align}
    \rho \colon \kfrak \longrightarrow \sorth(\wedge^i z_0).
\end{align}
The connection form $\rho(\thetil)$ in $\Omega^1(G(\R)^+ \times z_0,\wedge^j z_0)$ defines a covariant derivative 
\begin{align}
    \nabtil \colon \Omegtil^{0,j} \longrightarrow \Omegtil^{1,j}
\end{align}on $\wedge^j \Etil$.
We can extend it to a map
\begin{align}
    \nabtil \colon \Omegtil^{i,j} \longrightarrow \Omegtil^{i+1,j}
\end{align}
by setting
\begin{align}
    \nabtil(\omega \otimes s) \coloneqq d\omega \otimes s+(-1)^{i}\omega \wedge \nabtil(s),
\end{align}
as in \eqref{extconnec}. The connection on $\Omegtil^{i,j}$ is compatible with the metric. 
Finally, the covariant derivative $\nabtil$ defines a derivation on $\oplus \Omegtil^{i,j}$ that satisfies
\begin{align}
    \nabtil(\alpha \wedge \alpha')=(\nabtil \alpha) \wedge \alpha'+(-1)^{i+j}\alpha \wedge (\nabtil\alpha')
\end{align}
for any  $\alpha$ in $\Omegtil^{i,j}$ and $\alpha'$ in $\Omegtil^{k,l}$.

Taking the derivative of the tautological section gives an element 
\begin{align}
\nabtil{\bf s}=d{\bf s }+\rho(\thetil) {\bf s} \in \Omegtil^{1,1}.    
\end{align}
Let $\sorth(\Etil)$ denote the bundle $(G(\R)^+ \times z_0) \times_K \sorth(z_0)$ and consider the curvature $\rho(\widetilde{R})$ in $\Omega^2(\Etil, \sorth(\Etil))$.  We have an isomorphism 
 \begin{align} 
\restr{T^{-1}}{z_0} \colon \sorth(z_0) & \longrightarrow \wedge^2 z_0  \nonumber \\
A & \longmapsto \sum_{i<j}\langle Ae_i,e_j\rangle e_i \wedge e_j.
\end{align}
The inverse sends $v \wedge w$ to the endomorphism $ u \mapsto \langle v,u \rangle w-\langle w,u \rangle v$, and is the isomorphism from \eqref{eq:sowedge2} restricted to $z_0$. Note that we have
\begin{align} \label{anequality}
    T(v \wedge w) u =\iota(u) v \wedge w.
\end{align}
Using this isomorphism we can also identify $\sorth(\Etil)$ and $\wedge^2 \Etil$ so that we can view the curvature as an element
\begin{align}
\rho(\widetilde{R}) \in \Omegtil^{2,2}.  
\end{align} 

\begin{lem} The form $\omega \coloneqq 2 \pi \lVert {\bf s} \rVert^2+2 \sqrt{\pi}\nabtil {\bf s}-\rho(\Rtil) $ lying in $\Omegtil^{0,0} \oplus \Omegtil^{1,1} \oplus \Omegtil^{2,2}$ is annihilated by $\nabtil+ 2 \sqrt{\pi} i({\bf s})$. Moreover
\begin{align}d\int^B \alpha=\int^B \nabtil \alpha,\end{align}
for every form $\alpha$ in $\Omegtil^{i,j}$. Hence $\int^B exp(-\omega)$ is a closed form. 
\end{lem}
\begin{proof} We have
\begin{align}
    & \left (\nabtil+ 2 \sqrt{\pi} i({\bf s}) \right ) \left ( 2 \pi \lVert {\bf s} \rVert^2+2 \sqrt{\pi}\nabtil {\bf s}-\rho(\Rtil) \right ) \\
    & = 2 \pi \nabtil \lVert {\bf s} \rVert^2 + 4\pi^\frac{3}{2}i({\bf s}) \lVert {\bf s} \rVert^2 + 2 \sqrt{\pi} \nabtil^2 {\bf s}+4 \pi i(x) \nabtil {\bf s}-\nabtil \rho(\Rtil)-2 \sqrt{\pi} i({\bf s}) \rho(\Rtil). \nonumber
\end{align}
It vanishes because we have the following:
\begin{enumerate}
    \item[$\cdot$] $i({\bf s}) \lVert {\bf s} \rVert^2=0$ since $\lVert {\bf s} \rVert$ is in $\Omegtil^{0,0}$,
    \item[$\cdot$] $\nabtil \rho(\Rtil)=0$ by Bianchi's identity,
    \item[$\cdot$] $\nabtil\lVert {\bf s} \rVert^2= 2\langle \nabtil{\bf s},{\bf s} \rangle=-2 i({\bf s}) \nabtil{\bf s} $,
    \item[$\cdot$] $\nabtil^2{\bf s}=\rho(\Rtil){\bf s} =i({\bf s})\rho(\Rtil)$.
\end{enumerate}
For the last point we used \eqref{anequality} where we view $\rho(\Rtil)$ as an element of $\Omega^2(E,\sorth(\Etil))$, respectively of $\Omega^2(E,\wedge^2\Etil)$.

Let $s_1 \wedge \cdots \wedge s_q$ in $\Omega^0(E,\wedge^q \Etil)$ be a global section where $s_1, \dots, s_q$ is a local orthonormal frame for $\Etil$. Then for any $\alpha$ in $\Omegtil^{i,j}$ we have
\begin{align}\int^B\alpha=\langle \alpha, s_1 \wedge \cdots \wedge s_q \rangle. \end{align}
This vanishes if $j$ is different from $q$, hence we can assume $\alpha$ is in $\Omegtil^{i,q}$. If we write $\alpha$ as $\beta s_1 \wedge \cdots \wedge s_q$ for some $\beta$ in $\Omega^i(E)$ then
\begin{align}\int^B \alpha = \beta.\end{align}
On the other hand, since the connection on $\Omegtil^{i,q}$ is compatible with the metric, we have
\begin{align}
    0 = d \langle s_1 \wedge \cdots \wedge s_q,s_1 \wedge \cdots \wedge s_q \rangle = 2 \langle \nabtil(s_1 \wedge \cdots \wedge s_q),s_1 \wedge \cdots \wedge s_q \rangle.
\end{align}
Then we have
\begin{align}
    \int^B \nabtil \alpha & =\langle \nabtil \alpha, s_1 \wedge \cdots \wedge s_q \rangle   \nonumber \\
    & = \langle d\beta \otimes s_1 \wedge \cdots \wedge s_q + (-1)^i \beta \wedge \nabtil( s_1 \wedge \cdots \wedge s_q),s_1 \wedge \cdots \wedge s_q \rangle \nonumber \\
    & = d \beta \nonumber \\
    & = d \int^B \alpha .
\end{align}
Since $\nabtil + 2 \sqrt{\pi}i({\bf s})$ is a derivation that annihilates $\omega$ we have
\begin{align} \left (\nabtil + 2 \sqrt{\pi}i({\bf s}) \right ) \omega^k=0\end{align}
for positive $k$. Hence it follows that
\begin{align}
    d \int^B \exp(-\omega) & = \int^B \nabtil \exp(-\omega) \nonumber \\
    & = \int^B \left (\nabtil+2 \sqrt{\pi}i({\bf s}) \right )\exp(-\omega) \nonumber \\
    & =0.
\end{align}
\end{proof}

In \cite{mq} Mathai and Quillen define the following form
\begin{align}
U_{MQ} \coloneqq (- 1)^\frac{q(q+1)}{2} (2\pi)^{-\frac{q}{2}}\int^B \exp \left (-2 \pi \lVert {\bf s} \rVert^2-2 \sqrt{\pi} \nabtil {\bf s}+\rho(\Rtil) \right ) \in \Omega_{rd}^{q}(E).
\end{align}
We call it the {\em Mathai-Quillen form}.
\begin{prop} \label{mq}
The Mathai-Quillen form is a Thom form.
\end{prop}

\begin{proof}
From the previous lemma it follows that the form is closed. It remains to show that its integral along the fibers is $1$. The restriction of the form $U_{MQ}$ along the fiber $\pi^{-1}(eK)$ is given by
\begin{align}
U_{MQ} & = (- 1)^\frac{q(q+1)}{2} (2\pi)^{-\frac{q}{2}} e^{-2 \pi \lVert {\bf s} \rVert^2} \int^B \exp(-2 \sqrt{\pi} d {\bf s}) \nonumber \\ 
& = (- 1)^\frac{q(q+1)}{2} 2^{\frac{q}{2}} e^{-2 \pi \lVert {\bf s} \rVert^2} (-1)^q \int^B (dx_1 \otimes e_1) \wedge \cdots \wedge (dx_q \otimes e_q) \nonumber \\
& =  2^{\frac{q}{2}} e^{-2 \pi \lVert {\bf s} \rVert^2} dx_1 \wedge \cdots \wedge dx_q, 
\end{align}
and its integral over the fiber $\pi^{-1}(eK)$ is equal to $1$.
\end{proof}

\subsection{Transgression form}

For $t>0$ consider the map $t\colon  E \longrightarrow E$ given by multiplication by $t$ in the fibers. Consider the $K$-invariant vector field
\begin{align}
    X \coloneqq \sum_{i=1}^q x_i \frac{\partial}{\partial x_i}
\end{align} on $G(\R)^+ \times \R^q$. Since it is $K$-invariant it also induces a vector field on $E$. We define  the {\em transgression form} $\psi$ in $\Omega^{q-1}(E)$ to be $\iota_X U_{MQ}$, where $\iota_X$ is the interior product.

\begin{prop}[Transgression formula]
The transgression satisfies:
\begin{align} 
\left ( \frac{d}{dt} t^\ast U_{MQ} \right )_{t=t_0}=-\frac{1}{t_0}d (t_0^\ast \psi).
\end{align}
\end{prop}
\begin{proof} This is due to Mathai and Quillen. Let us view the multiplication map by $t$ as a map
\begin{align}
    m \colon E \times \R_{>0} & \longrightarrow E \nonumber \\
    (e,t) & \longmapsto et.
\end{align}
The differential $\tilde{d}$ on $E \times \R_{> 0}$ splits as $d+d_{\R_{>0}}$. Since $U_{MQ}$ is closed (hence its pullback) we have
\begin{align} \label{formula1}
0=\tilde{d}(m^\ast U_{MQ})=d(m^\ast U_{MQ})+\frac{d}{d t}(m^\ast U_{MQ}) dt.
\end{align}
Moreover the pushforward of the vector field $t\frac{\partial}{\partial t}$ by $m$ is $X$, hence for the contraction we have
\begin{align}
\iota_{\frac{\partial}{\partial t}} m^\ast U_{MQ}=\frac{1}{t}m^\ast \iota_X U_{MQ}.
\end{align}Since the differential $d$ is independent of $t$ it commutes with the contraction $\iota_{\frac{\partial}{\partial t}}$. Combining with \eqref{formula1} yields
\begin{align}
\frac{d}{d t}(m^\ast U_{MQ})=-\frac{1}{t}d(m^\ast \psi).\end{align}
Finally, pulling back by the section 
\begin{align}
t_0 \colon E & \longrightarrow E \times \R_{>0} \nonumber \\
e & \longmapsto (e,t_0)
\end{align} gives the desired formula.
\end{proof}

Let $\Gamma_v$ be the stabilizer of $v$ in $\Gamma$, which acts on the left on $E$. By the $G(\R)^+$-invariance (hence $\Gamma_v$-invariance) of $U_{MQ}$, it is also a form in $\Omega^q(\Gamma_v \backslash E)$. Let $S_0$ denote the image $\Gamma_v \backslash E_0$ of the zero section in $\Gamma_v \backslash E$.
\begin{prop} \label{dualS0} The form $U_{MQ}$ represents the Poincaré dual of $S_0$ in $\Gamma_v \backslash E$.
\end{prop}

\begin{proof}[Sketch of proof]
For $0<t_1<t_2$ we have
\begin{align}
t_2^\ast U_{MQ}-t_1^\ast U_{MQ} & =\int_{t_1}^{t_2}\left ( \frac{d}{dt} t^\ast U_{MQ} \right ) dt \nonumber \\
& =-\int_{t_1}^{t_2}d(t^\ast \psi)\frac{dt}{t} \nonumber\\
& = -d \int_{t_1}^{t_2}t^\ast \psi\frac{dt}{t} 
\end{align}
so that $t_2^\ast U_{MQ}$ and $t_1^\ast U_{MQ}$ represent the same cohomology class in $\cohom^q(\Gamma_v \backslash E)$. Then, one can show that 
\begin{align}
    \lim_{t \rightarrow \8} t^\ast U_{MQ} = \delta_{S_0}
\end{align}
where $\delta_{S_0}$ is the current of integration along $S_0$. Hence if $\omega$ is a form in $\Omega_c^{m}(E)$, where $m$ is the dimension of $\D^+$, then 
\begin{align}
    \int_{\Gamma_v \backslash E} U_{MQ} \wedge \omega & = \lim_{t \rightarrow \8} \int_{\Gamma_v \backslash E}  t^\ast U_{MQ} \wedge \omega \nonumber  \\
    & = \int_{S_0} \omega.
\end{align}
\end{proof}

\section{Computation of the Mathai-Quillen form} \label{sectioncomput}

\subsection{The section $s_v$} \label{sectionsv}

Let $\textrm{pr}$ denote the orthogonal projection of $V(\R)$ on the plane $z_0$. Consider the section
\begin{align}
s_v \colon \D^+ & \longrightarrow  E \nonumber  \\
z \; & \longmapsto [g_z,\textrm{pr}(g_z^{-1}v)],
\end{align}
 where $g_z$ is any element of $G(\R)^+$ sending $z_0$ to $z$. Let us denote by $L_g$ the left action of an element $g$ in $G(\R)^+$ on $\D^+$. We also denote by $L_g$ the action on $E$ given by $L_{g}[g_z,v]=[gg_z,v]$. The bundle is $G(\R)^+$-equivariant with respect to these actions.

\begin{prop} The section $s_v$ is well-defined and $\Gamma_v$-equivariant. Moreover its zero locus is precisely $\D^+_v$.
\end{prop}
\begin{proof}
The section is well-defined, since replacing $g_z$ by $g_zk$ gives 
\begin{align}s_v(z)=[g_zk,\textrm{pr}(k^{-1}g_z^{-1}v)]=[g_zk,k^{-1}\textrm{pr}(g_z^{-1}v)]=[g,\textrm{pr}(g_z^{-1}v)]=s_v(z).\end{align}
Suppose that $z$ is in the zero locus of $s_v$, that is to say $\textrm{pr}(g_z^{-1}v)$ vanishes. Then $g_z^{-1}v$ is in $z_0^\perp$. It is equivalent to the fact that $z=g_zz_0$ is a subspace of $v^\perp$, which means that $z$ is in $\D_v^+$. Hence the zero locus of $s_v$ is exactly $\D^+_v$.
For the equivariance, note that we have
\begin{align}
s_v \circ L_{g}(z)=[gg_z,\pr(g_z^{-1}g^{-1}v)]=L_{g} \circ s_{g^{-1}v}(z).  
\end{align}
Hence if $\gamma$ is an element of $\Gamma_v$ we have 
\begin{align}
 s_v \circ L_{\gamma}=L_\gamma \circ s_v.  
\end{align}
\end{proof} 

\noindent We define the pullback $\varphi^0(v) \coloneqq s_v^\ast U_{MQ}$ of the Mathai-Quillen form by $s_v$. It defines a form
\begin{align}
    \varphi^0 \in C^\8(\R^{p+q}) \otimes \Omega^q(\D)^+.
\end{align}
It is only rapidly decrasing on $\R^q$, and in order to make it rapidly decreasing everywhere we set
\begin{align}
    \varphi(v)\coloneqq e^{-\pi Q(v,v)}\varphi^0(v).
\end{align}
It defines a form $ \varphi \in \Scal(\R^{p+q}) \otimes \Omega^q(\D)^+$
\begin{prop} \begin{enumerate}
\item For fixed $v$ in $V(\R)$ the form $\varphi^0(v)$ in $\Omega^q(\D^+)$ is given by
\begin{align}
    \varphi^0(v) & = (- 1)^\frac{q(q+1)}{2} (2\pi)^{-\frac{q}{2}} \exp \left (2 \pi \restr{Q}{z_0}(v,v) \right ) \int^B \exp \left (-2 \sqrt{\pi} \nabla s_v+\rho(R)\right ).
\end{align} 
\item It satisfies $L_g^\ast \varphi^0(v)=\varphi^0(g^{-1}v)$, hence 
 \begin{align} \varphi^0 \in \left [\Omega^q(\D^+) \otimes C^\8(\R^{p+q}) \right ]^{G(\R)^+}.\end{align}
 \item It is a Poincaré dual of $\Gamma_v \backslash \D_v^+$ in $\Gamma_v \backslash \D^+$.
 \end{enumerate}
\end{prop}
\begin{proof}\begin{enumerate}
    \item  Recall that $\nabtil=\pi^\ast\nabla$ and $\Rtil=\pi^\ast R$. We pullback by $s_v$
\[
\begin{tikzcd}
E \simeq s_v^\ast \Etil \arrow[d] \arrow[r] & \Etil \arrow[d,"\pi"] \\
\D^+  \arrow[r,"s_v"] & E. 
\end{tikzcd}
\]
Since $\pi \circ s_v$ is the identity we have
\begin{align}
    s_v^\ast \nabtil = s_v^\ast \pi^\ast \nabla = \nabla.
\end{align}
Hence, the pullback connection $s_v^\ast \nabtil$ satisfies 
\begin{align}
s_v^\ast (\nabtil {\bf s})= (s_v^\ast \nabtil) ( s_v^\ast {\bf s}) = \nabla s_v
\end{align}
since $s_v^\ast {\bf s}=s_v$. We also have $s_v^\ast \Rtil=R$ and 
\begin{align}s_v^\ast \lVert {\bf s} \rVert^2= \lVert s_v \rVert^2= \langle s_v , s_v \rangle=-\restr{Q}{z_0}(v,v).  \end{align}
The expression for $\varphi^0$ then follows from the fact that $\exp$ and $s_v^\ast$ commute.
\item The bundle $E$ is $G(\R)^+$ equivariant. By construction the Mathai-Quillen is $G(\R)^+$-invariant, so $L_g^\ast U_{MQ}=U_{MQ}$. On the other hand we also have 
\begin{align}s_v \circ L_g(z)=L_g \circ s_{g^{-1}v}(z),\end{align}
and thus 
\begin{align}
    L_g^\ast \varphi^0(v)=L_g^\ast s_v^\ast U_{MQ}=\varphi^0(g^{-1}v).
\end{align}

\item  Since $s_v$ is $\Gamma_v$-equivariant we view it as a section
\begin{align}
    s_v \colon \Gamma_v \backslash \D^+ \longrightarrow \Gamma_v \backslash E,
\end{align}
whose zero locus is precisely $\Gamma_v \backslash \D_v^+$. Let $S_0$ (respectively $S_v$) be the image in $\Gamma_v \backslash E$ of the section $s_v$ (respectively the zero section). By Proposition \ref{dualS0} the Thom form $U_{MQ}$ is a Poincaré dual of $S_0$. For a form $\omega$ in $\Omega_c^{m-q}(\Gamma_v \backslash \D^+)$ we have
\begin{align}
    \int_{\Gamma_v \backslash \D^+} \varphi^0(v) \wedge \omega & = \int_{\Gamma_v \backslash \D^+} s_v^\ast \left ( U_{MQ} \wedge \pi ^\ast \omega \right ) \nonumber \\
    & = \int_{S_v} U_{MQ} \wedge \pi ^\ast \omega \nonumber \\
    & = \int_{S_v \cap S_0} \pi ^\ast \omega \nonumber \\
    & = \int_{\Gamma_v \backslash \D_v^+} \omega.
\end{align}
The last step follows from the fact $\pi^{-1}(S_v \cap S_0)$ equals $\Gamma_v \backslash \D_v^+$.
\end{enumerate}
\end{proof}

\noindent As in \eqref{isoate} we have an isomorphism
\begin{align}
    \left [ \Omega^q(\D^+) \otimes C^\8(\R^{p+q}) \right ]^{G(\R)^+} & \longrightarrow \left [ \sideset{}{^q}\bigwedge \p^\ast \otimes C^\8(\R^{p+q}) \right ]^K
\end{align}
by evaluating at the basepoint $eK$ of $G(\R)^+/K$ that corresponds to $z_0$ in $\D^+$. We will now compute $\restr{\varphi^0}{eK}$.

\subsection{The Mathai-Quillen form at the identity}
From now on we identify $\R^{p+q}$ with $V(\R)$ by the orthonormal basis of \eqref{basis}, and let $z_0$ be the negative spanned by the vectors $e_{p+1}, \cdots , e_{p+q}$. Hence we identify $z_0$ with $\R^q$ and the quadratic form is
\begin{align}
    \restr{Q}{z_0}(v,v)=-\sum_{\mu=p+1}^{p+q} x_\mu^2
\end{align}
where $x_{p+1}, \dots, x_{p+q}$ are the coordinates of the vector $v$. 

 Let $f_v$ in $\Omega^0(G(\R)^+,z_0)^K$ be the map associated to the section $s_v$, as in Proposition \ref{basicforms}. It is defined by 
 \begin{align}
     f_v(g)=\operatorname{pr}(g^{-1}v).
 \end{align} Then $df_v+\rho(\theta) f_v$ is the horizontal lift of $\nabla s_v$, as discussed in Section \ref{assbundle}. Let $X$ be a vector in $\g$ and let $X_\p$ and $X_\kfrak$ be its components with respect to the splitting of $\g$ as $\p \oplus \kfrak$. We have \begin{align}
    (df_v+\rho(\theta) f_v)_e(X)=d_ef_v(X_\p). 
 \end{align} In particular we can evaluate on the basis $X_{\alpha \mu}$ and get:
\begin{align}
d_{e} f_v(X_{\alpha \mu}) & = \left . \frac{d}{dt} \right \rvert_{t=0} f_v(\exp tX_{\alpha \mu})\nonumber \\
& = -\operatorname{pr}(X_{\alpha \mu}v) \nonumber \\
& = -\operatorname{pr}(x_\mu e_\alpha+x_\alpha e_\mu) \nonumber\\
& = -x_{\alpha}e_\mu.
\end{align}
So as an element of $\p^\ast \otimes z_0$ we can write
\begin{align}d_ef_v=-\sum_{\mu=p+1}^{p+q} \left ( \sum_{\alpha=1}^p  x_\alpha \omega_{\alpha \mu} \right )\otimes e_\mu=-\sum_{\alpha=1}^p x_\alpha \eta_\alpha,\end{align}
with 
\begin{align}
\eta_\alpha \coloneqq \sum_{\mu=p+1}^{p+q} \omega_{ \alpha \mu} \otimes e_\mu \in \Omega^{1,1}.
\end{align}

\begin{prop} Let $\rho(R_e)$ in $\wedge^2\p^\ast \otimes \sorth(z_0)$ be the curvature at the identity. Then after identifying $\sorth(z_0)$ with $\wedge^2 z_0$ we have
\begin{align}
\rho(R_e)= -\frac{1}{2} \sum_{\alpha=1}^p \eta_\alpha^2 \in \wedge^2\p^\ast \otimes \wedge^2 z_0,
\end{align}
where $\eta_\alpha^2=\eta_\alpha \wedge \eta_\alpha$.
\end{prop}

\begin{proof} Using the relation $E_{ij}E_{kl}=\delta_{il}E_{kj}$ one can show that
\begin{align}
    [X_{\alpha \mu},X_{\beta \nu}] =\delta_{\mu \nu} X_{\alpha \beta}+ \delta_{\alpha \beta} X_{\mu \nu}
\end{align}
for two vectors $X_{\alpha \nu}$ and $ X_{\beta \mu}$ in $\p$. Hence we have
\begin{align}
R_e(X_{\alpha \nu} \wedge X_{\beta \mu}) & =[\theta(X_{\alpha \nu}), \theta(X_{\beta \mu})] -\theta([X_{\alpha \nu}, X_{\beta \mu}]) \nonumber \\
& =-\theta([X_{\alpha \nu}, X_{\beta \mu}]) \nonumber \\
& =-p\left (\delta_{\alpha \beta}X_{\nu \mu}+\delta_{\nu \mu}X_{\alpha \beta} \right ) \nonumber \\
& = -\delta_{\alpha \beta}X_{\nu \mu}.
\end{align}
On the other hand, since $\eta_i(X_{jr})=\delta_{ij}e_r$, we also have
\begin{align}
\sum_{i=1}^p \eta_i^2(X_{\alpha \nu} \wedge X_{\beta \mu}) & = \sum_{i=1}^p \eta_i(X_{\alpha \nu}) \wedge \eta_i(X_{\beta \mu})-\eta_i(X_{\beta \mu}) \wedge \eta_i(X_{\alpha \nu}) \nonumber \\
& =2 \delta_{\alpha \beta}e_\nu \wedge e_\mu.
\end{align}
The lemma follows since $\rho(X_{\nu \mu})=T(e_\nu \wedge e_\mu)$ in $\sorth(z_0)$, because
\begin{align}Q( \rho(X_{\nu \mu})e_\nu, e_\mu ) e_\nu \wedge e_\mu =-Q( e_\mu, e_\mu ) e_\nu \wedge e_\mu=e_\nu \wedge e_\mu.\end{align}

\end{proof}
Using the fact that the exponential satisfies $\exp(\omega+\eta)=\exp(\omega)\exp(\eta)$ on the subalgebra $\bigoplus \Omega^{i,i}$ - see Remark \ref{commutative} - we can write
\begin{align} \label{intermed1}
\restr{\varphi^0}{e}(v) &= (- 1)^\frac{q(q+1)}{2} (2\pi)^{-\frac{q}{2}} \exp \left (2 \pi \restr{Q}{z_0}(v,v) \right ) \int^B \prod_{\alpha=1}^p \exp \left ( 2 \sqrt{\pi} x_\alpha\eta_\alpha-\frac{1}{2}\eta_\alpha^2 \right ).
\end{align}

We define the {\em $n$-th Hermite polynomial} by
\begin{align}
    H_n(x) \coloneqq \left ( 2x-\frac{d}{dx} \right ) \cdot 1 \in \R[x].
\end{align}
The first three Hermite polynomials are $H_0(x)=1$, $H_1(x)=2x$ and $H_2(x)=4x^2-2$.
\begin{lem} Let $\eta$ be a form in $\bigoplus \Omega^{i,i}$. Then
\begin{align}\exp(2x \eta-\eta^2)=\sum_{n \geq 0} \frac{1}{n!}H_n(x)\eta^n, \end{align}
where $H_n$ is the $n$-th Hermite polynomial.
\end{lem}

\begin{proof} Since $\eta$ and $\eta^2$ are in $\bigoplus \Omega^{i,i}$, they commute and we can use the binomial formula:
\begin{align}
\exp(2x \eta-\eta^2) & =\sum_{k \geq 0} \frac{1}{k!} \left (2x\eta-\eta^2 \right)^k \nonumber \\
& = \sum_{k \geq 0} \frac{1}{k!} \sum_{l=0}^k {k \choose l}(2x \eta)^{k-l}\left (-\eta^2 \right)^{l} \nonumber \\
& = \sum_{k \geq 0} \frac{1}{k!} \sum_{l=0}^k {k \choose l}(2x)^{k-l} (- 1 )^{l} \eta^{l+k} \nonumber \\
& = \sum_{n \geq 0} P_n(x)\eta^n,
\end{align}
where
\begin{align}
P_n(x)\coloneqq \sum_{\substack{ 0 \leq l \leq k \leq n \\ k+l=n} } \frac{(- 1 )^{l}}{l! (k-l)!}(2x)^{k-l}.
\end{align}
 The conditions on $k$ and $l$ imply that $n$ is less than or equal to $2k$. First suppose that $n$ is even. Then we have that $k$ is between $\frac{n}{2}$ and $n$, so that the sum above can be written
\begin{align}
\sum_{k=\frac{n}{2}}^{n}\frac{ (-1)^{n-k}}{(n-k)!(2k-n)!}(2x)^{2k-n}=&\sum_{m=0}^{\frac{n}{2}}\frac{(-1)^{\frac{n}{2}-m}}{(\frac{n}{2}-m)!(2m)!} (2x)^{2m}=\frac{1}{n!}H_n(x),
\end{align}
where in the second step we let $m$ be $k-\frac{n}{2}$. If $n$ is odd then  $k$ is between $\frac{n+1}{2}$ and $n$, so that the sum can be written
\begin{align}
\sum_{k=\frac{n+1}{2}}^{n}\frac{ (-1)^{n-k}}{(n-k)!(2k-n)!}(2x)^{2k-n}=&\sum_{m=0}^{\frac{n-1}{2}}\frac{(-1)^{\frac{n-1}{2}-m}}{(\frac{n-1}{2}-m)!(2m+1)!} (2x)^{2m+1}=\frac{1}{n!}H_n(x).
\end{align}
\end{proof}

Applying the lemma to \eqref{intermed1} we get
\begin{align}
& \int^B \prod_{\alpha=1}^p \exp \left (2 \sqrt{\pi} x_\alpha\eta_\alpha-\frac{1}{2}\eta_\alpha^2 \right ) \nonumber \\
&=  \int^B \prod_{\alpha=1}^p \exp \left (2 \sqrt{2\pi}x_\alpha \frac{\eta_\alpha}{\sqrt{2}}-\left (\frac{\eta_\alpha}{\sqrt{2}}\right)^2 \right ) \nonumber \\
& = \int^B \prod_{\alpha=1}^p \sum_{n \geq 0} \frac{2^{-n/2}}{n!}H_n\left (\sqrt{2\pi} x_\alpha \right )\eta_\alpha^n \nonumber \\ 
& = \sum_{n_1, \dots, n_p}\frac{2^{-\frac{n_1+\cdots +n_p}{2}}}{n_1! \cdots n_p!}H_{n_1}\left (\sqrt{2\pi} x_1 \right ) \cdots H_{n_p}\left (\sqrt{2\pi}x_p \right ) \int^B \eta_1^{n_1} \wedge \cdots \wedge \eta_p^{n_p}.
\end{align}
If $n_1+ \cdots + n_p$ is different from $q$, then the Berezinian of $\eta_1^{n_1} \wedge \cdots \wedge \eta_p^{n_p}$ vanishes and we get
\begin{align}
& \sum_{n_1, \dots, n_p}\frac{2^{-\frac{n_1+\cdots +n_p}{2}}}{n_1! \cdots n_p!}H_{n_1}\left (\sqrt{2\pi} x_1 \right ) \cdots H_{n_p}\left (\sqrt{2\pi}x_p \right )\int^B \eta_1^{n_1} \wedge \cdots \wedge \eta_p^{n_p} \nonumber \\
= & 2^{-\frac{q}{2}}\sum_{n_1+ \dots+ n_p=q}\frac{H_{n_1}\left (\sqrt{2\pi}x_1 \right ) \cdots H_{n_p}\left (\sqrt{2\pi} x_p \right )}{n_1! \cdots n_p!} \int^B \eta_1^{n_1} \wedge \cdots \wedge \eta_p^{n_p}.
\end{align}
Note that
\begin{align}
\eta_\alpha^{n_\alpha} & = \left ( \sum_{\mu=p+1}^{p+q} \omega_{\alpha \mu} \otimes e_\mu \right )^{n_\alpha} \nonumber  \\
& = \sum_{\mu_1, \dots, \mu_{n_\alpha}} (\omega_{\alpha \mu_1} \otimes e_{\mu_1}) \wedge \cdots \wedge (\omega_{\alpha \mu_{n_\alpha}} \otimes e_{\mu_{n_\alpha}}) \nonumber \\
& = n_\alpha! \sum_{\mu_1< \dots< \mu_{n_\alpha}}(\omega_{\alpha \mu_1} \otimes e_{\mu_1}) \wedge \cdots \wedge (\omega_{\alpha \mu_{n_\alpha}} \otimes e_{\mu_{n_\alpha}}),
\end{align}
where the sums are over all $\mu_i$'s between $p+1$ and $p+q$. If $n_1+\cdots+n_p$ is equal to $q$  we have
\begin{align}
& \int^B \eta_1^{n_1} \wedge \cdots \wedge \eta_p^{n_p} \nonumber \\
& = \int^B \prod_{\alpha=1}^{p} \left ( \sum_{\mu=p+1}^{p+q} \omega_{\alpha \mu} \otimes e_\mu \right )^{n_\alpha} \nonumber \\
& = \int^B \prod_{\alpha=1}^{p} n_\alpha! \sum_{\mu_1< \dots< \mu_{n_\alpha}}(\omega_{\alpha \mu_1} \otimes e_{\mu_1}) \wedge \cdots \wedge (\omega_{\alpha \mu_{n_\alpha}} \otimes e_{\mu_{n_\alpha}}) \nonumber \\
& =   n_1! \cdots n_p! \sum \int^B (\omega_{\alpha (p+1)} \otimes e_{1}) \wedge \cdots \wedge (\omega_{\alpha (p+q)} \otimes e_{q}) \nonumber \\
& = (- 1)^\frac{q(q+1)}{2}  n_1! \cdots n_p! \sum \omega_{\alpha_1 (p+1)} \wedge \cdots \wedge \omega_{\alpha_q (p+q)},
\end{align}
where the sums in the last two lines go over all tuples $\underline{\alpha}=(\alpha_1, \dots, \alpha_q)$ with $\alpha$ between $1$ and $p$, and the value $\alpha$ appears exactly $n_{\alpha}$-times in $\underline{\alpha}$. Hence
\begin{align}
 \restr{\varphi^0}{e}(v) = \begin{multlined}[t] 2^{-q}\pi^{-\frac{q}{2}}\sum \omega_{\alpha_1 (p+1)} \wedge \cdots \wedge \omega_{\alpha_q (p+q)} \otimes H_{n_1}\left (\sqrt{2\pi} x_1 \right ) \\
 \cdots H_{n_p}\left (\sqrt{2\pi} x_p \right )\exp \left (2\pi \restr{Q}{z_0}(v,v) \right ). \end{multlined}
\end{align}
After multiplying by $\exp \left (-\pi Q(v,v) \right )$ we get
\begin{align}
\restr{\varphi}{e}(v) = \begin{multlined}[t] 2^{-q}\pi^{-\frac{q}{2}} \sum \omega_{\alpha_1 (p+1)} \wedge \cdots \wedge \omega_{\alpha_q (p+q)} \otimes H_{n_1}\left (\sqrt{2\pi} x_1 \right )
\\ 
\cdots H_{n_p}\left (\sqrt{2\pi} x_p \right )\exp \left (-\pi Q_{z_0}^+(v,v) \right ). \end{multlined}
\end{align}
The form is now rapidly decreasing in $v$, since the Siegel majorant is positive definite. We have
\begin{align}
    \restr{\varphi}{e} \in \left [ \sideset{}{^q} \bigwedge \p^\ast \otimes \Scal(\R^{p+q})\right ]^K.
\end{align}
\begin{thm} \label{mainthm} We have $2^{-\frac{q}{2}}\varphi(v)=\varphi_{KM}(v)$.
\end{thm}
\begin{proof}
It is a straightforward computation to show that
\begin{align}
(2 \pi)^{-n_\alpha/2}H_{n_\alpha}\left ( \sqrt{2 \pi}x_\alpha \right )\exp(-\pi x_\alpha^2)  & =   \left (x_\alpha - \frac{1}{2\pi}\frac{\partial}{\partial{x_\alpha}} \right )^{n_\alpha} \exp(-\pi x_\alpha^2).
\end{align}
Hence applying this we find that the Kudla-Millson form, defined by the Howe operators in \eqref{kmdef}, is
\begin{align}
    \restr{\varphi_{KM}}{e}(v) & = \begin{multlined}[t]
    2^{-q}(2\pi)^{-\frac{q}{2}} \sum \omega_{\alpha_1 (p+1)} \wedge \cdots \wedge \omega_{\alpha_q (p+q)} \otimes H_{n_1}\left (\sqrt{2\pi} x_1 \right ) \\
    \cdots H_{n_p}\left (\sqrt{2\pi} x_p \right )\exp \left (-\pi \restr{Q}{z_0}(v,v) \right ) 
    \end{multlined} \nonumber \\
    & = 2^{-\frac{q}{2}} e^{-\pi Q(v,v)}\restr{\varphi^0}{e}(v).
\end{align}
\end{proof}

\section{Examples}
\begin{enumerate}
\item Let us compute the Kudla-Millson as above in the simplest setting of signature $(1,1)$. Let $V(\R)$ be the quadratic space $\R^2$ with the quadratic form $Q(v,w)=x'y+xy'$ where $x$ and $x'$ (respectively $y$ and $y'$) are the components of $v$ (respectively of $w$). Let $e_1=\frac{1}{\sqrt{2}}(1,1)$ and $e_2=\frac{1}{\sqrt{2}}(1,-1)$. The $1$-dimensional negative plane $z_0$ is $\R e_2$. If $r$ denotes the variable on $z_0$ then the quadratic form is $\restr{Q}{z_0}(r)=-r^2$. The projection map is given by
\begin{align}
    \pr \colon V(\R) & \longrightarrow z_0 \nonumber \\
    v=(x,x') & \longmapsto \frac{x-x'}{\sqrt{2}}.
\end{align}
The orthogonal group of $V(\R)$ is
\begin{align}
    G(\R)^+=\left \{ \begin{pmatrix} t & 0 \\ 0 & t^{-1} \end{pmatrix},t>0 \right \},
\end{align}
and $\D^+$ can be identified with $\R_{>0}$. The associated bundle $E$ is $\R_{>0} \times \R$ and the connection $\nabla$ is simply $d$ since the bundle is trivial. Hence the Mathai-Quillen form is 
\begin{align}
    U_{MQ}=\sqrt{2}e^{-2\pi r^2}dr \in \Omega^1(E),
\end{align}
as in the proof of Proposition \ref{mq}. The section $s_v \colon \R_{>0} \rightarrow E$ is given by
\begin{align}
    s_v(t)=\left (t, \frac{t^{-1}x-tx'}{\sqrt{2}} \right ),
\end{align}
where $x$ and $x'$ are the components of $v$. We obtain
\begin{align}
    s_v^\ast U_{MQ}=e^{-\pi \left ( \frac{x}{t}-tx' \right )^2}\left ( \frac{x}{t}+tx' \right ) \frac{dt}{t}.
\end{align}
Hence after multiplication by $2^{-\frac{1}{2}}e^{-\pi Q(v,v)}$ we get
\begin{align}
    \varphi_{KM}(x,x')= 2^{-\frac{1}{2}}e^{-\pi \left [ \left (\frac{x}{t} \right )^2+(tx')^2 \right ]}\left ( \frac{x}{t}+tx' \right ) \frac{dt}{t}
\end{align}

\item The second example illustrates the functorial properties of the Mathai-Quillen form. Suppose that we have an orthogonal splitting of $V(\R)$ as $\bigoplus_i^r V_i(\R)$. Let $(p_i,q_i)$ be the signature of $V_i(\R)$. We have
\begin{align}
    \D_{1} \times \cdots \times \D_{r} \simeq \left \{ z \in \D \; \vert \; z = \bigoplus_{i=1}^r z \cap V_i(\R) \right \} .
\end{align} Suppose we fix $z_0= z_0^1 \oplus \cdots \oplus z_0^r$ in $\D^+_{1} \times \cdots \times \D^+_{r} \subset \D$, where $z_0^i$ is a negative $q_i$-plane in $V_i(\R)$. Let $G_i(\R)$ be the subgroup preserving $V_i(\R)$, let $K_i$ the stabilizer of $z_0^i$ and $\D_i$ be the symmetric space associated to $V_i(\R)$. 

Over $\D^+_{1} \times \cdots \times \D^+_{r}$ the bundle $E$ splits as an orthogonal sum $E_1 \oplus \cdots \oplus E_r$, where 
$E_i$ is the bundle $G_i(\R)^+ \times_{K_i} z_0^i$. Moreover the restriction of the Mathai-Quillen form to this subbundle is
\begin{align}
    \restr{U_{MQ}}{E_1 \times \cdots \times E_r}=U_{MQ}^1 \wedge \cdots \wedge U_{MQ}^r,
\end{align}
where $U_{MQ}^i$ is the Mathai-Quillen form on $E_i$. The section $s_v$ also splits as a direct sum $\oplus s_{v_i}$ where $v_i$ is the projection of $v$ onto $v_i$. In summary the following diagram commutes
\begin{equation}
\begin{tikzcd}
E_1 \oplus \cdots \oplus E_r \arrow[hook, r] & E \\
\D^+_1 \times \cdots \times \D^+_r \arrow[u,"\oplus s_{v_i}"] \arrow[hook, r] & \D^+ \arrow[u,"s_v"]
\end{tikzcd},
\end{equation}
and we can conclude that
\begin{align}
    \restr{\varphi_{KM}(v)}{\D_1^+ \times \cdots \times \D_r^+}=\varphi_{KM}^1(v_1)\wedge  \cdots \wedge \varphi_{KM}^r(v_r)
\end{align}
where $\varphi_{KM}^i$ is the Kudla-Millson form on $\D_i^+$.
\end{enumerate}

\printbibliography

\end{document}